\documentclass[a4paper,12pt]{article}
\usepackage{latexsym}
\usepackage{amsmath}
\usepackage{amssymb}
\usepackage{enumerate}
\usepackage{graphicx}

\usepackage{theorem}
\newtheorem{theorem}{Theorem}[section]

\newtheorem{lemma}[theorem]{Lemma}

\newtheorem{question}{Question.}

\theorembodyfont{\rmfamily}
\newtheorem{proof}{\textmd{\textit{Proof.}}}

\newtheorem{remark}[theorem]{Remark}

\newtheorem{acknowledgement}{\textmd{\textit{Acknowledgements.}}}

\makeatletter

\@addtoreset{equation}{section}
\makeatother

\newcommand{\qedd}{\hfill \Box}
\newcommand{\ve}{\varepsilon}
\newcommand{\lra}{\longrightarrow}

\newcommand{\wt}{\widetilde}

\newcommand{\ol}{\overline}
\newcommand{\B}{\ensuremath{\mathbb{B}}}

\newcommand{\R}{\ensuremath{\mathbb{R}}}

\newcommand{\Sph}{\ensuremath{\mathbb{S}}}

\newcommand{\cU}{\ensuremath{\mathcal{U}}}

\def\ra{\mathop{\mathrm{rank}}\nolimits}

\def\vol{\mathop{\mathrm{vol}}\nolimits}

\def\supp{\mathop{\mathrm{supp}}\nolimits}

\def\Cut{\mathop{\mathrm{Cut}}\nolimits}
\def\bL{\mathop{\mathrm{Lip^b}}\nolimits}

\def\Im{\mathop{\mathrm{Im}}\nolimits}
\def\Ker{\mathop{\mathrm{Ker}}\nolimits}

\title{Differentiable sphere theorems whose 
comparison\\[1mm]
spaces are standard spheres or exotic ones
}

\author{Kei Kondo\footnote{Department of Mathematical Sciences, Yamaguchi University, Yamaguchi City, Yamaguchi Pref. 753-8512, Japan. 
{\em email address:} {\tt keikondo@yamaguchi-u.ac.jp}} \, and \,Minoru Tanaka\footnote{Department of Mathematics, Tokai University, 
Hiratsuka City, Kanagawa Pref. 259-1292, Japan. 
{\em email address:} {\tt tanaka@tokai-u.jp}}
}
\date{\today}
\begin{document}
\maketitle

\begin{abstract}
We show that for an arbitrarily given closed Riemannian manifold $M$ admitting a point $p \in M$ 
with a single cut point, every closed Riemannian 
manifold $N$ admitting a point 
$q \in N$ with a single cut point is diffeomorphic to $M$ if the radial curvatures of $N$ at $q$ are sufficiently 
close in the sense of $L^1$-norm to those of 
$M$ at $p$. Our result hence not only 
produces a weak version of the Cartan--Ambrose--Hicks 
theorem in the case where underlying manifolds 
admit a point with a single cut point, but 
also is a kind of a weak version of 
the Blaschke conjecture for spheres proved 
by Berger. 
In particular that result generalizes one of theorems 
in Cheeger's Ph.D. Thesis in that case. 
Remark that every exotic sphere of dimension 
$> 4$ admits a metric such that 
there is a point whose cut locus consists 
of a single point.
\end{abstract}

\section{Introduction}\label{sec1}

In the global Riemannian geometry 
the relationship between curvatures and structures, especially topology, of Riemannian manifolds 
has been studied from various kinds of viewpoint, 
and a great number of results concerning with 
such a relation has been gotten. It is the topological 
$1/4$--pinching sphere theorem 
that is counted among the masterpieces of such results 
from the geodesic theory's standpoint, which states the following: 

\begin{theorem}{\rm (Rauch--Berger--Klingenberg)}\label{2019_01_01_thm1.1} 
If a compact simply connected Riemannian 
manifold $X$ admits a metric whose sectional 
curvature $K_X$ lies in $(1/4,1]$, 
then the manifold is homeomorphic to a sphere. 
\end{theorem}
This masterpiece was very first proved 
by Rauch \cite{Ra} in the case where $3/4< K_X \le 1$, 
and worked out by Berger \cite{Berg} 
and Klingenberg \cite{Klin} 
in the case where $1/4< K_X \le 1$. 
Note that the complex projective space 
admits a metric satisfying 
$1/4 \le K_X \le 1$ and is not homeomorphic 
to a sphere.\par 
Theorem \ref{2019_01_01_thm1.1} 
produced the $1/4$--pinching race as the problem 
if ``homeomorphic" in the statement 
could be replaced by ``diffeomorphic". 
There were a large number of entrant 
for the race, e.g., Gromoll \cite{Grom}, Calabi, 
Shikata \cite{Sh2}, Sugimoto--Shiohama--Karcher \cite{SSK}, 
Grove--Karcher--Ruh \cite{GKR1, GKR2}, 
Im Hof--Ruh \cite{IR}, and Suyama \cite{Suy}, et al. 
Using the Ricci flow introduced by 
Hamilton \cite{Ham}, 
Brendle and Schoen \cite{BrenSch} finally proved 
that the masterpiece can be reinforced into the differentiable $1/4$--pinching sphere theorem, 
which implies that every exotic sphere does not admit a 
$1/4$--pinched metric.\par 
By remembering that 
the $1/4$--pinching race (problem) had originated in 
Hopf's curvature pinching conjecture, the solution 
to the problem by Brendle--Schoen asks the following 
natural question of us. 

\begin{question}{\rm Replacing the unit standard sphere 
in the Hopf conjecture by an arbitrary compact 
simply connected Riemannian manifold $X$, 
should a compact 
simply connected Riemannian manifold whose 
radial curvature is close to that of $X$ 
be diffeomorphic to $X$? That is, can we weaken the assumption of the Cartan--Ambrose--Hicks 
theorem \cite{Amb, Cartan} to closeness 
of radial curvatures of the manifolds?}
\end{question}
In the question above 
the radial curvature is, by definition, 
the restriction of the sectional curvature of 
a pointed Riemannian manifold to all $2$-dimensional planes which contain the unit tangent velocity vector, 
as one of its basis, of any minimal geodesic emanating from the base point. 

\bigskip

The purpose of this article is {\em to solve the 
question above by hypothesizing that 
underlying 
closed manifolds admit metrics such that there 
is a point whose cut locus consists of a single point}. 
It is worthy of note that every homotopy $n$-sphere of 
dimension $n\ge5$ admits such a metric, 
and so are all exotic $n$-spheres. This note follows 
from Smale's $h$-cobordism theorem \cite{Sm1, Sm2} 
and Weinstein's deformation 
technique \cite{W} for metrics on twisted spheres 
(also see \cite[Proposition 7.19]{Bess}). 

\bigskip

We are now going to state our main theorem precisely. 
For each $k=1,2$ let $M_k$ be a closed 
manifold 
of dimension $n\ge2$ admitting a point 
whose cut locus consists of a single point, and 
$\langle \, \cdot\,, \cdot\, \rangle$ a Riemannian metric of 
$M_k$. 
Note that $M_k$ is homeomorphic to a sphere $S^n$ 
of dimension $n$. We take any point 
$p_k \in M_k$ satisfying 
$\Cut (p_k)=\{q_k\}$ 
where $q_k \in M_k$, and fix it. Here 
$\Cut (p_k)$ denotes the cut locus of $p_k$. 
Normalizing the metric, we can assume  
$d_{M_k}(p_k,q_k)= \pi$ 
where $d_{M_k}$ denotes the distance 
function of $M_k$. 
Set 
$
\Sph_{p_k}^{n-1}:=\{u \in T_{p_k}M_k\,|\,\|u\|=1\}
$ 
where $T_{p_k}M_k$ is the tangent space to $M_k$ 
at $p_k$. 
For each $u_k\in \Sph_{p_k}^{n-1}$, 
let $\tau_{u_k}:[0,\pi]\lra M_k$ 
denote a geodesic segment emanating 
from $p_k=\tau_{u_k}(0)$ to $q_k=\tau_{u_k}(\pi)$ 
in the direction 
$
u_k = \dot{\tau}_{u_k} (0):= (d\tau_{u_k}/dt) (0)
$, i.e., 
\begin{equation}\label{2017_04_01_geod}
\tau_{u_k}(t)=\exp_{p_k}t u_k
\end{equation} 
for all $t \in [0,\pi]$.\par 
Fix $u_1 \in \Sph^{n-1}_{p_1}$. 
Let $I_{p_1}: T_{p_1}M_1\lra T_{p_2}M_2$ be 
a linear isometry. Set $u_2:= I_{p_1}(u_1)$. 
For each $t \in [0,\pi]$ let $P^{(u_k)}_t$ $(k=1,2)$ 
be the parallel translation along the geodesic 
$\tau_{u_k}$ from $p_k$ to $\tau_{u_k}(t)$ where note that 
$\tau_{u_2} =\tau_{I_{p_1}(u_1)}$. 
Define the linear isometry 
$\Psi^{(u_1)}_t:T_{p_1}M_1\lra T_{\tau_{u_2}(t)}M_2$ 
by
\begin{equation}\label{2017_04_01_isom}
\Psi^{(u_1)}_t:= P^{(u_2)}_t \circ I_{p_1}. 
\end{equation}
Moreover we define the function $\lambda:[0,\pi]\lra \R$ by 
\begin{equation}\label{2017_05_14_fn}
\lambda (t)
:=\max_{
\substack{
u_1 \,\in\, \Sph_{p_1}^{n-1}
\\[0.5mm] x_1 \,\in\, \Sph_{u_1}^{n-2}}
} 
\left|
K^{(1)} (P^{(u_1)}_t (x_1)\wedge P^{(u_1)}_t (u_1))
- K^{(2)} (\Psi^{(u_1)}_t (x_1)\wedge \Psi^{(u_1)}_t (u_1))
\right|
\end{equation}
for all $t \in [0,\pi]$ where $
\Sph_{u_1}^{n-2}
:=
\{
x \in T_{u_1} (\Sph_{p_1}^{n-1}) \,|\,\|x\|=1
\}$ and $K^{(k)} (x\wedge y)$ ($k=1,2$) is 
the sectional curvature of the plane spanned by two 
linearly independent tangent vectors $x$ and $y$ 
at a point on $M_k$, i.e., 
\begin{equation}\label{2017_05_04_def_of_sec_curv}
K^{(k)} (x\wedge y)=\frac{\left\langle 
R^{(k)}\left(x,y\right)x, y 
\right\rangle
}{\|x\|^2\|y\|^2-\left\langle x, y \right\rangle^2}.
\end{equation}
In Eq.\,\eqref{2017_05_04_def_of_sec_curv}, $R^{(k)}$ denotes 
the curvature tensor of $M_k$ defined by 
\[
R^{(k)} (X, Y)Z
:= 
\nabla_Y \nabla_X Z - \nabla_X \nabla_Y Z + \nabla_{[X,\,Y]}Z
\]
for all vector fields $X, Y, Z$ on $M_k$, 
where $\nabla$ is the Levi--Civita connection on $M_k$. 
Note that the definition of 
$R^{(k)}$ differs from that 
of curvature tensor in literatures such as 
\cite{CE} and \cite{Sa} by a sign. 

\bigskip

With the these notations above 
our main theorem is stated as follows. 
 
\begin{theorem}{\rm (Main Theorem)}\label{2017_04_14_thm}
There exists a constant $\ve_n(M_1) >0$ 
depending on $n$ and $M_1$ such that if 
\begin{equation}\label{2017_04_14_thm_1}
\int^\pi_0\lambda (t)dt < \ve_n(M_1),
\end{equation}
then $M_2$ is diffeomorphic to $M_1$. 
\end{theorem}

\begin{remark}\label{2017_05_26_rem}
We give here several remarks on 
Theorem \ref{2017_04_14_thm} and 
related results to it:
\begin{itemize}
\item Theorem \ref{2017_04_14_thm} is 
the very Cartan--Ambrose--Hicks theorem if 
$\lambda (t)\equiv 0$ on $[0,\pi]$, 
and hence produces a weak version of the theorem. 
Note here that sectional curvature 
and curvature tensor are equivalent (see, e.g., 
Eq.\,\eqref{2017_05_04_lemma1_3}). 
Moreover, since $M_k$ ($k =1,2$) 
admits a point whose cut locus consists of a single point, Theorem \ref{2017_04_14_thm} 
is a kind of a weak version of the Blaschke conjecture for spheres proved 
by Berger \cite{Bess}, 
which states that if a Riemannian manifold 
$N$ of dimension $n\ge2$ homeomorphic 
to $S^n$ has diameter equal to its 
injectivity radius, then $N$ 
is isometric to a standard sphere of constant 
curvature. In particular our theorem generalizes Cheeger's theorem \cite{C} 
(or see \cite[Theorem 7.36]{CE}) in the case where underlying manifolds 
admit a point with a single cut point, 
because we do not assume either 
closeness of $\nabla R^{(1)}$ and 
$\nabla R^{(2)}$ along $\tau_{u_1}$ 
and $\tau_{u_2}$ or $\vol (M_2) > \nu$ for some $\nu>0$ 
where $\vol (M_2)$ denotes the volume of $M_2$,  
that additionally he assumed in his theorem; 
besides, we need to look around such manifolds 
only at their base points $p_1$ and $p_2$. 
Moreover it is apparent that our theorem extends and 
weakens \cite[(iii) of Theorem 3]{KK} 
to a wider class of metrics than 
that of radially symmetric metrics in it.

\item The constant $\ve_n(M_1)$ in 
Eq.\,\eqref{2017_04_14_thm_1} is obtained as 
the unique solution of the following equation \begin{equation}\label{2017_05_26_rem_1}
\eta_n(M_1)\cdot x \exp (2(n-1)x) 
=
\frac{1}{2}
\left[
\sqrt{1+ \Big\{ \frac{8}{\pi}(n-1)\Big\}^{-\frac{1}{2}}} -1
\right]
\end{equation}
for all $x \in [0,\infty)$ where $\eta_n(M_1)$, 
depending on $n$ and $M_1$, denotes some positive 
constant concerning with Jacobi fields along $\tau_{u_1}$ 
(see Eq.\,\eqref{2017_05_06_lem2_1} for 
more details). The constant $1+\{(8/\pi)(n-1)\}^{-1/2}$
found in Eq.\,\eqref{2017_05_26_rem_1} is the same as Karcher \cite{K} estimated in order to prove a sharper version of Shikata's theorem in \cite{Sh1}.
\item The related results to 
Theorem \ref{2017_04_14_thm} are the differentiable exotic sphere theorems I and II proved by authors \cite{KT}. In the theorem I, the hypothesis \eqref{2017_04_14_thm_1} 
can be replaced by either 
\begin{equation}\label{2017_05_26_rem_2}
\Big\| \frac{d^{2} c_{\,\gamma}}{dt^2} \Big\|^2 - 
2\bL (\sigma)^{-2} 
\le 2\Big\{ \frac{\sqrt{2}-1}{2(e^\pi-1)} \Big\}^2 -1
\end{equation}
for all unit speed geodesic segments 
$\gamma ([0,\pi])\subset \Sph_{q_1}^{n-1}:=\{v \in T_{q_1}M_1\,|\,\|v\|=1\}$, 
or   
$
\bL(\sigma)^2 \le 1+ \left\{ (8/\pi)(n-1)\right\}^{-1/2}
$ 
where
$
\sigma: \Sph^{n-1}_{q_1}
\lra \Sph^{n-1}_{q_2}
:=\{v \in T_{q_2}M_2\,|\,\|v\|=1\}$ 
is the diffeomorphism defined 
by Eq.\,\eqref{2017_04_18_Sigma}, 
$\bL(\sigma)$ is the bi-Lipschitz constant 
of $\sigma$ defined by 
\[
\bL(\sigma)
:=\inf\{
a \,|\, a^{-1} \|u-v\| 
\le 
\|\sigma (u)-\sigma(v)\|
\le a \|u-v\| \ {\rm for \ all} \ u,v \in \Sph^{n-1}_{q_1}\}, 
\]
and 
$c_{\,\gamma}:= \sigma\circ \gamma : [0,\pi]\lra\Sph^{n-1}_{q_2}$ for each geodesic 
segment $\gamma ([0,\pi])\subset \Sph_{q_1}^{n-1}$. 
In the theorem II, the hypothesis \eqref{2017_05_26_rem_2} is replaced by 
$
\angle(\ol{c}_{\,\gamma}(t), c_{\,\gamma}(t)) < 
\pi / 2$ 
for all unit speed geodesic segments 
$\gamma ([0,\pi])\subset \Sph_{q_1}^{n-1}$ where $\ol{c}_{\,\gamma}$ is the smooth curve on $T_{q_2}M_2$ given by 
$\ol{c}_{\,\gamma}(t):= c_{\,\gamma}(0)\cos t + (dc_{\,\gamma}/dt) (0)\sin t$ for all $t \in [0,\pi]$.
\end{itemize}
\end{remark}

\begin{acknowledgement}
In this work the first named author was supported by the JSPS KAKENHI Grant Numbers 17K05220, and partially 
16K05133, 18K03280.
\end{acknowledgement}

\section{
Key lemma
}\label{sec2}

The aim of this section is to show Key Lemma (Lemma \ref{2017_04_16_Key_LEMMA}) that we shall apply to the proof of Theorem \ref{2017_04_14_thm}. The integral form 
of the Gr\"onwall inequality \cite{Bell, Gron} 
plays an important role in the proof 
of Lemma \ref{2017_04_16_Key_LEMMA}.\par  
Throughout this section, 
for each $k=1,2$ let $M_k$ be a closed manifold 
of dimension $n\ge2$ admitting a point $p_k \in M_k$ 
such that $\Cut (p_k)=\{q_k\}$, 
and for any $u_k \in \Sph^{n-1}_{p_1}$ let 
$\tau_{u_k}:[0,\pi]\lra M_k$ 
be the geodesic segment emanating from $p_k$ to $q_k$ defined by Eq.\,\eqref{2017_04_01_geod}. 
We assume $d_{M_k}(p_k,q_k):= \pi$ by normalizing the metric where $d_{M_k}$ denotes 
the distance function on $M_k$. 
All other notations in the following are the same 
as those defined in Section \ref{sec1}.\par  
Fix $u_1 \in \Sph^{n-1}_{p_1}$. 
Choose an orthonormal basis 
$e_1^{(u_1)}, e_2^{(u_1)}, \dots, e_n^{(u_1)}$ 
of $T_{p_1}M_1$ satisfying $e_n^{(u_1)} = u_1$. 
Setting $u_2:=I_{p_1}(u_1)$, 
we have the orthonormal basis 
$e_1^{(u_2)}, e_2^{(u_2)}, \dots, e_n^{(u_2)}$ 
of $T_{p_2}M_2$ given by 
$e_i^{(u_2)}:= I_{p_1} ( e_i^{(u_1)})$ for each 
$i=1,2, \ldots, n$, which satisfies $e_n^{(u_2)}= u_2$. 
For each $k=1,2$ let
\begin{equation}\label{2017_05_05_sec2_1}
E_i^{(k)} (t):=P^{(u_k)}_t ( e_i^{(u_k)}), \quad i=1,2, \ldots, n.
\end{equation}
$E_1^{(k)}, E_2^{(k)}, \dots, E_n^{(k)}$ 
are then the parallel orthonormal fields along $\tau_{u_k}$. 
In particular 
\begin{equation}\label{2017_05_05_sec2_2}
E_i^{(1)}(t) = P_t^{(u_1)}(e^{(u_1)}_i), 
\quad 
E_i^{(2)}(t) = P_t^{(u_2)}(I_{p_1}(e^{(u_1)}_i)) 
= \Psi_t^{(u_1)} (e^{(u_1)}_i)
\end{equation}
for each $i=1,2,\ldots,n-1$ and 
\begin{equation}\label{2017_05_05_sec2_3}
\dot{\tau}_{u_1} (t)= E_n^{(1)}(t)= P_t^{(u_1)} (u_1), \quad 
\dot{\tau}_{u_2} (t)= E_n^{(2)}(t)= \Psi_t^{(u_1)} \big(u_1\big).
\end{equation}
Moreover for each $i,j=1,2,\ldots,n-1$ 
let  
\[
a_{ij}^{(k)} (t)
:=
\big\langle 
R^{(k)}
(
E_i^{(k)} (t), \dot{\tau}_{u_k}(t)
)
\dot{\tau}_{u_k}(t), 
E_j^{(k)} (t)
\big\rangle
\] 
for all $t\in[0,\pi]$. 
Furthermore we define 
the square matrix 
$A (t\,; u_k)$ of order $2(n-1)$ by 
\[
A (t\,; u_k):=\left(
\begin{array}{cc}
0 & I_{n-1}\\
\big(a_{ij}^{(k)} (t)\big) & 0
\end{array}
\right),
\]
where $I_{n-1}$ is the $(n-1)$-th unit matrix. 
Note that $\big(a_{ij}^{(k)} (t)\big)$ is the 
symmetric matrix of order $n-1$. 

\begin{lemma}\label{2017_05_04_lemma1}
For any $t \in [0,\pi]$,
\[
\|A(t;u_1)-A(t;u_2)\|\le 2(n-1)\lambda (t)
\]
holds where $\|\cdot\|$ denotes the linear 
operator norm. In particular
\[
\|A(t;u_2)\|\le c_1 (n, p_1) + 2(n-1)\lambda (t)
\]
for all $t \in [0,\pi]$ 
where 
$c_1 (n, p_1)
:= 
\max
\{\|A(t;u_1)\|\,|\, 
u_1 \in \Sph_{p_1}^{n-1}, 
t \in [0,\pi]
\}$.
\end{lemma}

\begin{proof}
Fix $t \in [0,\pi]$. Since 
\begin{equation}\label{2017_04_15_sec2.2_1}
a_{ij}^{(k)} (t)
=
- \big\langle 
R^{(k)} 
(
E_i^{(k)} (t), \dot{\tau}_{u_k}(t)
)
E_j^{(k)} (t), 
\dot{\tau}_{u_k}(t)
\big\rangle \quad (k=1,2), 
\end{equation}
Eq.\,\eqref{2017_05_04_def_of_sec_curv} gives 
\begin{equation}\label{2017_05_04_lemma1_1}
K^{(k)} ( E^{(k)}_i(t) \wedge \dot{\tau}_{u_k}(t))
= \big\langle 
R^{(k)} 
(
E^{(k)}_i(t), \dot{\tau}_{u_k}(t)
)
E^{(k)}_i(t), \dot{\tau}_{u_k}(t)
\big\rangle 
= -a_{ii}^{(k)} (t).
\end{equation}
Since $E_i^{(k)}(t)+E_j^{(k)}(t)$ is orthonormal to $\dot{\tau}_{u_k}(t)$, and since 
$\|E_i^{(k)}(t)+E_j^{(k)}(t)\| = \sqrt{2}$, 
combining Eqs.\,\eqref{2017_05_04_def_of_sec_curv}, 
\eqref{2017_04_15_sec2.2_1}, and \eqref{2017_05_04_lemma1_1} shows 
\begin{equation}\label{2017_05_04_lemma1_2}
K^{(k)} ((E_i^{(k)}(t)+E_j^{(k)}(t)) \wedge \dot{\tau}_{u_k}(t))
= \frac{1}{2}
\big(
-a_{ii}^{(k)} (t)-2a_{ij}^{(k)} (t)-a_{jj}^{(k)} (t)
\big).
\end{equation}
From Eqs.\,\eqref{2017_05_04_lemma1_1} and 
\eqref{2017_05_04_lemma1_2} we have 
\begin{align}\label{2017_05_04_lemma1_3}
a_{ij}^{(k)} (t) = 
&\frac{1}{2}
\big\{
K^{(k)} (E^{(k)}_i(t)\wedge \dot{\tau}_{u_k}(t)) 
+K^{(k)} (E^{(k)}_j(t)\wedge \dot{\tau}_{u_k}(t)) 
\big\}\notag\\[1mm]
& - K^{(k)} ((E_i^{(k)}(t)+E_j^{(k)}(t)) \wedge \dot{\tau}_{u_k}(t)).
\end{align}
Since $K^{(k)}(x\wedge y)$ does not depend 
on the choice of the spanning vectors, we see,  
by Eqs.\,\eqref{2017_05_05_sec2_2}, \eqref{2017_05_05_sec2_3}, \eqref{2017_05_04_lemma1_3} and 
the triangle inequality, that 
\begin{equation}\label{2017_05_04_lemma1_4}
\big|a_{ij}^{(1)} (t)-a_{ij}^{(2)} (t) \big| \le 2 \lambda (t).
\end{equation}
We therefore see, by Eq.\,\eqref{2017_05_04_lemma1_4}, that 
\begin{align*}
\|A(t;u_1)-A(t;u_2)\| 
&=\big\|
\big(
a_{ij}^{(1)} (t)-a_{ij}^{(2)} (t)
\big)
\big\|\\[1mm]
&\le(n-1)
\max_{i,\,j\,=\,1,\,2,\,\ldots,\,n-1}
\big| a_{ij}^{(1)} (t)-a_{ij}^{(2)} (t) \big|\\[1mm]
&\le 2(n-1)\lambda (t),
\end{align*}
which is the first assertion. Since 
\[
\|A(t;u_2)\| = \|A(t;u_2)-A(t;u_1)+A(t;u_1)\| 
\le \|A(t;u_2)-A(t;u_1)\|+\|A(t;u_1)\|,
\]
the second assertion follows from the first one. 
$\qedd$
\end{proof}

Let $\Sph^{n-1}_{q_k}:=\{v \in T_{q_k}M_k\,|\,\|v\|=1\}$ ($k=1,2$). 
Since $\Cut (p_k)=\{q_k\}$, 
we have the diffeomorphism 
$\sigma^{p_k}_{q_k}$ from $\Sph^{n-1}_{p_k}$ 
onto $\Sph^{n-1}_{q_k}$ 
given by
\begin{equation}\label{2017_05_14_diffeom}
\sigma^{p_k}_{q_k} (u_k):= -\dot{\tau}_{u_k} (\pi),  
\end{equation}
for all $u_k \in \Sph^{n-1}_{p_k}$. 
The map $\sigma$ from $\Sph^{n-1}_{q_1}$ onto 
$\Sph^{n-1}_{q_2}$ defined by
\begin{equation}\label{2017_04_18_Sigma}
\sigma
:= \sigma^{p_2}_{q_2} \circ I_{p_1} \circ \sigma^{q_1}_{p_1}
\end{equation}
is thus a diffeomorphism where 
$\sigma^{q_1}_{p_1}:= (\sigma_{q_1}^{p_1})^{-1}$. 
Moreover for each $u_1 \in \Sph^{n-1}_{p_1}$ 
let $I_{q_1}^{(u_1)}:T_{q_1}M_1\lra T_{q_2}M_2$ 
denote the linear isometry given by 
\begin{equation}\label{2017_05_10_Isom}
I_{q_1}^{(u_1)} :=\Psi_\pi^{(u_1)}\circ (P_\pi^{(u_1)})^{-1}.
\end{equation}

\begin{lemma}{\rm (Key Lemma)}\label{2017_04_16_Key_LEMMA}
Set
\begin{equation}\label{2017_04_16_sec2.2_new2}
\delta_n:= 
\sqrt{1+ 
\Big\{
\frac{8}{\pi}(n-1)
\Big\}^{-\frac{1}{2}}} -1.
\end{equation}
Then 
there is a constant $\ve_n(M_1) >0$ 
such that if 
\[
\int_0^\pi \lambda (t) dt < \ve_n(M_1),
\]
then for any $u_1 \in \Sph^{n-1}_{p_1}$ 
the differential $d \sigma_{v_1}$ of $\sigma$ at 
$v_1:=\sigma^{p_1}_{q_1}(u_1)\in \Sph^{n-1}_{q_1}$
and $I^{(u_1)}_{q_1}$ are $\delta_n/2$-close with respect to the linear 
operator norm.
\end{lemma} 

\begin{proof}
Fix $u_1 \in \Sph^{n-1}_{p_1}$.  Let 
$u_2:=I_{p_1}(u_1) \in \Sph^{n-1}_{p_2}$, 
and $\Sph_{u_2}^{n-2}
:=\{x \in T_{u_2}(\Sph_{p_2}^{n-1})\,|\,\|x\|=1\}$.  
For any fixed $x_1 \in \Sph_{u_1}^{n-2}$ let 
$x_2:= I_{p_1}(x_1) \in \Sph_{u_2}^{n-2}$ 
where we identify $T_{u_k}(T_{p_k}M_k)$ 
with $T_{p_k}M_k$ ($k=1,2$). 
Let $J^{\,(x_k)}_{u_k}$ be 
the Jacobi field along $\tau_{u_k}$ given by 
$J^{\,(x_k)}_{u_k}(0)=0$ and $(DJ^{\,(x_k)}_{u_k}/dt)(0)=x_k \in \Sph_{u_k}^{n-2}$ 
where $DJ^{\,(x_k)}_{u_k}/dt$ denotes the covariant 
derivative of $J^{(x_k)}_{u_k}$ along $\tau_{u_k}$. 
Setting 
$
f^{\,(k)}_i(t)
:= 
\langle 
J^{\,(x_k)}_{u_k} (t), E_i^{(k)} (t) 
\rangle
$ for all $t \in [0,\pi]$,  
the definition of $J^{\,(x_k)}_{u_k}$ 
and the Gauss lemma give 
$
J^{\,(x_k)}_{u_k} (t)
= \sum_{i=1}^{n-1}f^{\,(k)}_i(t) E_i^{(k)} (t)$. 
For simplicity of notation we set  
\[
\dot{f}^{\,(k)}_i (t) := \frac{d f^{\,(k)}_i}{dt}(t), \quad 
\ddot{f}^{\,(k)}_i (t) := \frac{d^{\,2} f^{\,(k)}_i}{dt^2}(t)
\]
for all $i=1,2,\ldots, n-1$. Since 
\[
(d\sigma^{p_k}_{q_k})_{u_k} (x_k)
=(d\sigma^{p_k}_{q_k})_{u_k} 
\Big(
\cfrac{DJ^{\,(x_k)}_{u_k}}{dt} (0)
\Big) 
=-\cfrac{DJ^{\,(x_k)}_{u_k}}{dt} (\pi),
\]
Eq.\,\eqref{2017_04_18_Sigma} shows 
\begin{align}\label{2017_04_02_lem1_3}
d \sigma_{v_1} 
\Big(
\cfrac{DJ^{\,(x_1)}_{u_1}}{dt} (\pi)
\Big) 
&=(d \sigma^{p_2}_{q_2}\circ I_{p_1})_{u_1} (-x_1)
=-(d \sigma^{p_2}_{q_2})_{u_2} (x_2)\notag
\\[1mm]
&=\frac{DJ^{\,(x_2)}_{u_2}}{dt} (\pi)
=\sum_{i=1}^{n-1} \dot{f}^{\,(2)}_i (\pi)E^{(2)}_i(\pi).
\end{align}
From Eq.\,\eqref{2017_05_05_sec2_1} and the second 
one in Eq.\,\eqref{2017_05_05_sec2_2} we obtain 
\[
I_{q_1}^{(u_1)} (E_i^{(1)} (\pi)) 
= \Psi_\pi^{(u_1)} (e^{(u_1)}_i) 
= E^{(2)}_i(\pi). 
\]
We then have 
\begin{equation}\label{2017_04_02_lem1_4}
I^{(u_1)}_{q_1}
\Big(
\cfrac{DJ^{\,(x_1)}_{u_1}}{dt} (\pi)
\Big) 
=I^{(u_1)}_{q_1}
\Big(
\sum_{i=1}^{n-1} \dot{f}^{\,(1)}_i (\pi)E^{(1)}_i(\pi)
\Big)
=\sum_{i=1}^{n-1} \dot{f}^{\,(1)}_i (\pi)E^{(2)}_i(\pi).
\end{equation}
It follows from Eqs.\,\eqref{2017_04_02_lem1_3} and \eqref{2017_04_02_lem1_4} that 
\begin{align}\label{2017_04_02_lem1_5}
\Big\|
d \sigma_{v_1} 
\Big(
\frac{DJ^{\,(x_1)}_{u_1}}{dt} (\pi)
\Big) 
-
I^{(u_1)}_{q_1}
\Big(
\frac{DJ^{\,(x_1)}_{u_1}}{dt} (\pi)
\Big) 
\Big\|
=\Big\| 
\sum_{i=1}^{n-1}
(
\dot{f}^{\,(2)}_i (\pi)- \dot{f}^{\,(1)}_i (\pi)
)
E^{(2)}_i(\pi)
\Big\|.
\end{align}

Let 
$
\wt{J}^{\,(x_k)}_{u_k}(t) 
:= {}^t 
(
f^{\,(k)}_1 (t), \ldots, f^{\,(k)}_{n-1} (t),
\dot{f}^{\,(k)}_1 (t), \ldots, \dot{f}^{\,(k)}_{n-1} (t)
)
\in \R^{2(n-1)}$. 
Define the smooth function $\varphi:[0,\pi]\lra\R$ by 
$
\varphi (t)
:=
\|
\wt{J}^{\,(x_1)}_{u_1} (t)- \wt{J}^{\,(x_2)}_{u_2} (t)
\|
$. 
In the case where $\varphi\equiv 0$ on $[0,\pi]$,
$\dot{f}^{\,(1)}_i (\pi)= \dot{f}^{\,(2)}_i (\pi)$ 
holds for each $i=1,2,\ldots,n-1$, and hence 
Eq.\,\eqref{2017_04_02_lem1_5} shows that 
$d \sigma_{v_1}$ and $I^{(u_1)}_{q_1}$ 
are $\delta_n/2$-close. From this argument 
we next consider the case where 
there is an interval $[a, b) \subset [0,\pi]$ 
such that $\varphi (t) >0$ on $(a,b)$ 
with $\varphi (a)=0$: Since $J^{\,(x_k)}_{u_k}$ 
satisfies the Jacobi equation 
\[
\cfrac{D^2 J^{\,(x_k)}_{u_k}}{dt^2} (t) 
+ 
R^{(k)} 
(
\dot{\tau}_{u_k}(t), J^{\,(x_k)}_{u_k} (t) 
)\dot{\tau}_{u_k}(t)=0, \quad t \in [0,\pi], 
\]
we obtain
\begin{equation}\label{2017_04_02_lem1_1}
\ddot{f}^{\,(k)}_j (t) - \sum^{n-1}_{i=1}a_{ij}^{\,(k)}(t)
f^{\,(k)}_i (t) =0 
\end{equation}
for all $t \in [0,\pi]$ and $j=1,2,\ldots, n-1$. 
Substituting Eq.\,\eqref{2017_04_02_lem1_1} 
for the following 
\[
\cfrac{d\wt{J}^{\,(x_k)}_{u_k}}{dt} (t)
={}^t 
(
\dot{f}^{\,(k)}_1 (t), \ldots, \dot{f}^{\,(k)}_{n-1} (t), 
\ddot{f}^{\,(k)}_1 (t), \ldots, \ddot{f}^{\,(k)}_{n-1} (t)
),
\]
we have
\[
\cfrac{d\wt{J}^{\,(x_k)}_{u_k}}{dt} (t)
= A(t\,;u_k)\wt{J}^{\,(x_k)}_{u_k}(t) \in \R^{2(n-1)}
\]
for all $t \in[0,\pi]$. Hence, 
applying the Cauchy--Schwarz inequality and the triangle one to $\varphi'(t)$, 
we see that 
for any $t \in (a,b)$, 
\begin{align}\label{2017_04_02_lem1_6}
\varphi'(t) 
&=
\frac{1}
{\| 
\wt{J}^{\,(x_1)}_{u_1} (t)- \wt{J}^{\,(x_2)}_{u_2} (t) 
\|}
\Big\langle 
\cfrac{d \wt{J}^{\,(x_1)}_{u_1}}{dt} (t)
- \cfrac{d \wt{J}^{\,(x_2)}_{u_2}}{dt} (t), 
\wt{J}^{\,(x_1)}_{u_1} (t)-\wt{J}^{\,(x_2)}_{u_2} (t)
\Big\rangle\notag\\[1mm]
&\le 
\|
A(t\,;u_1)\wt{J}^{\,(x_1)}_{u_1}(t)- A(t\,;u_2)\wt{J}^{\,(x_2)}_{u_2}(t)
\|\notag\\[1mm]
&=
\|
\left(A(t\,;u_1)- A(t\,;u_2)\right)
\wt{J}^{\,(x_1)}_{u_1}(t)
+ A(t\,;u_2)
(
\wt{J}^{\,(x_1)}_{u_1}(t) - \wt{J}^{\,(x_2)}_{u_2}(t)
)
\|\notag\\[1mm]
&\le
\|
A(t\,;u_1)- A(t\,;u_2)
\|
\cdot
\|
\wt{J}^{\,(x_1)}_{u_1}(t)
\|
+ 
\|
A(t\,;u_2)
\|
\cdot \varphi (t).
\end{align}
Let 
\[
h_a (t) := \int_a^t 
\|
A(s\,;u_1)- A(s\,;u_2)
\|
\cdot
\|
\wt{J}^{\,(x_1)}_{u_1}(s)
\|ds
\]
for all $t\in (a,b)$. 
Since $\varphi (a) =0$, 
the integration of Eq.\,\eqref{2017_04_02_lem1_6} 
from $a$ to $t$ yields the inequality 
\[
\varphi (t) 
\le 
h_a(t) 
+ 
\int_a^t 
\|
A(s\,;u_2)
\|
\cdot \varphi (s)\,ds.
\]
Since $\varphi (t)$, $h_a (t)$, 
and $\left\|A(t\,;u_2)\right\|$ are 
continuous on $(a,b)$, 
and since $\left\|A(t\,;u_2)\right\|\ge 0$ on $(a,b)$, 
the integral 
form of Gr\"onwall's 
inequality \cite{Gron, Bell} gives 
\begin{equation}\label{2017_04_02_lem1_7}
\varphi (t) 
\le 
h_a(t) 
+ 
\int_a^t 
h_a(s) 
\left\|
A(s\,;u_2)
\right\|
\exp 
\Big(\int_s^t \left\|
A(r\,;u_2)
\right\|dr
\Big)ds
\end{equation}
for all $t \in (a,b)$. 
Since $h_a$ is non-decreasing on $(a,b)$, 
we see, by Eq.\,\eqref{2017_04_02_lem1_7}, 
that 
\begin{align}
\varphi (t) 
&\le 
h_a(t) 
+ 
h_a(t) \int_a^t 
\left\|
A(s\,;u_2)
\right\|
\exp 
\Big(\int_s^t \left\|
A(r\,;u_2)
\right\|dr
\Big)ds\notag\\[1mm]
&= 
h_a(t) 
+ 
h_a(t) \int_a^t 
\frac{\partial}{\partial s}
\Big\{
-\exp 
\Big(
-\int^s_t 
\left\|
A(r\,;u_2)
\right\|dr
\Big)
\Big\}ds\notag\\[1mm]
&= 
h_a(t) 
+ 
h_a(t) 
\Big[
-\exp 
\Big(-\int^s_t \left\|
A(r\,;u_2)
\right\|dr
\Big)
\Big]_{s=a}^{s=t}\notag\\[1mm]
&= 
h_a(t) 
+ 
h_a(t) 
\Big\{
-1 + \exp 
\Big(\int_a^t \left\|
A(r\,;u_2)
\right\|dr
\Big)
\Big\}\notag\\[1mm]
&=h_a(t) 
\exp 
\Big(\int_a^t \left\|
A(r\,;u_2)
\right\|dr
\Big),\quad t \in (a,b).
\end{align}
Since the functions 
$
\left\|A(t\,;u_1)- A(t\,;u_2)\right\|
\cdot
\| \wt{J}^{\,(x_1)}_{u_1}(t)\|
$ 
and $\|A(t\,;u_2)\|$ are well-defined on $[0,b)$ and are integrable on $[0,b)$, 
it is clear that 
\begin{equation}\label{2017_04_02_lem1_9}
\varphi (t) 
\le 
h_a(t) 
\exp 
\Big(\int_a^t \left\|
A(r\,;u_2)
\right\|dr
\Big)
\le 
h_0(t) 
\exp 
\Big(\int_0^t \left\|
A(r\,;u_2)
\right\|dr
\Big) 
\end{equation}
for all $t \in (a,b)$. 
Since the function 
$
t \lra h_0(t) 
\exp 
\big(\int_0^t \left\|
A(r\,;u_2)
\right\|dr
\big)
$ 
is increasing on $[0,\pi]$, 
Eq.\,\eqref{2017_04_02_lem1_9} has the form
\begin{equation}\label{2017_04_02_lem1_10}
\varphi (t) 
\le 
h_0(\pi) 
\exp 
\Big(\int_0^\pi \left\|
A(r\,;u_2)
\right\|dr
\Big)
\end{equation}
for all $t \in (a,b)$. 
Since Eq.\,\eqref{2017_04_02_lem1_10} 
still holds for some $t_0 \in [0,\pi]$ with 
$\varphi (t_0)=0$, 
we get 
\begin{equation}\label{2017_04_02_lem1_11}
\varphi (t) 
\le 
h_0(\pi) 
\exp 
\Big(\int_0^\pi \left\|
A(r\,;u_2)
\right\|dr
\Big) 
\end{equation}
for all $t \in [0,\pi]$. Set 
$
c_2(n,p_1)
:= \max
\{
\| 
\wt{J}^{\,(x_1)}_{u_1} (t)
\|\,|\, 
u_1 \in \Sph_{p_1}^{n-1}, \, 
x_1 \in \Sph_{u_1}^{n-2}, \, t \in [0,\pi]
\}$. 
Applying 
Lemma \ref{2017_05_04_lemma1}
to Eq.\,\eqref{2017_04_02_lem1_11}, 
we have
\begin{align}\label{2017_04_02_lem1_13}
\varphi (t) 
&\le
h_0(\pi) \exp 
\Big(\int_0^\pi \left\|
A(r\,;u_2)
\right\|dr
\Big)\notag
\\[1mm]
&=
\int_0^\pi 
\|
A(r\,;u_1)- A(r\,;u_2)
\|
\cdot
\|
\wt{J}^{\,(x_1)}_{u_1}(r)
\|dr
\cdot \exp 
\Big(\int_0^\pi \left\|
A(r\,;u_2)
\right\|dr
\Big) 
\notag\\[1mm]
&\le
2(n-1)\cdot c_2(n, p_1) \cdot\exp (\pi c_1(n,p_1))
\int_0^\pi \lambda (r)dr
\cdot \exp 
\Big(
\int_0^\pi 2(n-1) \lambda (r)dr
\Big) 
\notag\\[1mm]
&= c(n, M_1) \int_0^\pi \lambda (r)dr
\cdot \exp 
\Big(
2(n-1) \int_0^\pi \lambda (r)dr
\Big).
\end{align}
for all $t \in [0,\pi]$ where 
$
c(n, M_1):= 2(n-1)\cdot c_2(n, p_1) \cdot\exp (\pi c_1(n,p_1))$. 
Let $\ve_n(M_1) >0$ be the unique solution of the following equation 
\begin{equation}\label{2017_05_06_lem2_1}
\frac{c(n, M_1)\cdot x \exp (2(n-1)x)}{c_3(n,p_1)} = \frac{1}{2}\delta_n 
\end{equation}
for all $x \in [0,\infty)$ where 
\[
c_3(n, p_1)
:=\min
\Big\{
\Big\|
\frac{D J^{\,(x_1)}_{u_1}}{dt} (\pi)
\Big\|\ \Bigl| \  
u_1 \in \Sph_{p_1}^{n-1}, \ 
x_1 \in \Sph_{u_1}^{n-2}
\Big\}.
\]
We now assume that 
\[
\int_0^\pi\lambda (s)ds < \ve_n (M_1).
\]
Eq.\,\eqref{2017_04_02_lem1_13} then 
yields 
\begin{equation}\label{2017_05_06_lem2_2}
\varphi (\pi) < c(n, M_1)\cdot \ve_n (M_1) \cdot 
\exp (2(n-1)\ve_n (M_1)) = \frac{c_3(n, M_1)}{2}\delta_n.
\end{equation}
Combining Eqs.\,\eqref{2017_04_02_lem1_5} and 
\eqref{2017_05_06_lem2_2} shows 
\begin{align}\label{2017_04_02_lem1_14}
\frac{
\Big\|
d \sigma_{v_1} 
\Big(
\frac{DJ^{\,(x_1)}_{u_1}}{dt}(\pi)
\Big) 
-I^{(u_1)}_{q_1} 
\Big(
\frac{DJ^{\,(x_1)}_{u_1}}{dt}(\pi)
\Big) 
\Big\|
}
{
\Big\|
\frac{DJ^{\,(x_1)}_{u_1}}{dt}(\pi)
\Big\|
}
&\le
\frac{1}{c_3(n,p_1)}\ 
\sqrt{
\sum_{i=1}^{n-1}
\big(
\dot{f}^{\,(2)}_i (\pi)- \dot{f}^{\,(1)}_i (\pi)
\big)^2
}\notag\\[3mm]
&\le
\frac{\varphi (\pi)}
{c_3(n,p_1)} < \frac{1}{2}\delta_n.
\end{align}
From the arbitrariness of $x_1$, 
Eq.\,\eqref{2017_04_02_lem1_14} implies  
\[
\big\|
d \sigma_{v_1} - I^{(u_1)}_{q_1} 
\big\|
=\sup_{
\substack{
w \,\in \,T_{v_1}(\Sph^{n-1}_{q_1}) \\[0.5mm] w\,\not=\,0
}
} 
\frac{\big\|
d \sigma_{v_1}(w) - I^{(u_1)}_{q_1} (w) 
\big\|}{\|w\|}
\le \frac{1}{2}\delta_n,
\]
and hence $d \sigma_{v_1}$ and $I^{(u_1)}_{q_1}$ 
are $\delta_n/2$-close. 
$\qedd$
\end{proof}

\section{
Proof of Theorem \ref{2017_04_14_thm}
}\label{sec3}

The purpose of this section is to show Theorem \ref{2017_04_14_thm}. The idea of the proof is to 
construct a family of smooth immersions 
which approximates a bi-Lipschitz homeomorphism, defined by Eq.\,\eqref{2018_12_14_Bi_Lip}, between 
closed Riemannian manifolds admitting a point 
with a single cut point under the assumption 
of Theorem \ref{2017_04_14_thm}. 
Throughout the section, for each $k=1,2$ 
let $M_k$ be a closed manifold of 
dimension $n\ge2$ admitting a point $p_k \in M_k$ 
such that $\Cut (p_k)=\{q_k\}$. Moreover we assume 
$d_{M_k}(p_k,q_k)= \pi$. Furthermore 
we will make the following assumption:
\[
\int_0^\pi \lambda (t) dt < \ve_n(M_1)
\]
where $\lambda (t)$ is the function defined 
by Eq.\,\eqref{2017_05_14_fn} and 
$\ve_n(M_1) >0$ is the unique solution 
of Eq.\,\eqref{2017_05_06_lem2_1}.\par 
Choose a linear isometry $I_{p_1}: T_{p_1}M_1\lra T_{p_2}M_2$, and we define the bi-Lipschitz homeomorphism 
$F$ from $M_1$ onto $M_2$ by 
\begin{equation}\label{2018_12_14_Bi_Lip}
F(\exp_{p_1}tu_1):=\exp_{p_2}(t I_{p_1}(u_1))
\end{equation}
for all $(t, u_1) \in [0,\pi] \times \Sph^{n-1}_{p_1}$. 
Since $\Cut (p_1)=\{q_1\}$, 
$F$ is not differentiable only at $q_1$, and $F|_{M_1\setminus \{q_1\}}$ is diffeomorphism. 
Let $\B_\pi (o_{q_k}):=\{x \in T_{q_k}M_k\,|\, \|x\|<\pi\}$ ($k=1,2$) where $o_{q_k}$ is the origin of $T_{q_k}M_k$. 
We define the map $\wt{F}: \B_\pi(o_{q_1})\lra \B_\pi(o_{q_2})$ by 
\[
\wt{F}:= \exp_{q_2}^{-1}\circ F \circ \exp_{q_1}.
\]
We then see, by the very same argument as \cite[Section 3.3]{KT}, that 
\[
\wt{F} (x)= 
\begin{cases}
\ \|x\| \sigma \Big(\dfrac{x}{\|x\|}\Big) \ &\text{for all} \ 
x \in \B_\pi(o_{q_1}) \setminus \{o_{q_1} \},\\[1mm]
\ o_{q_2}  \ &\text{for} \ x= o_{q_1}
\end{cases}
\]
where $\sigma: \Sph^{n-1}_{q_1}\lra \Sph^{n-1}_{q_2}$ 
is the diffeomorphism defined by
Eq.\,\eqref{2017_04_18_Sigma}. Note that $\wt{F}$ is a bi-Lipschitz homeomorphism (see \cite[Lemma 3.13]{KT} 
for more details). 

\begin{lemma}\label{2017_05_10_lem1}
For any $x \in \B_\pi(o_{q_1}) \setminus \{o_{q_1}\}$ and 
any $X \in T_x \B_\pi (o_{q_1})$ with $\|X\|=1$, 
\begin{equation}\label{2017_05_09_lem1_7}
1 -\frac{1}{2}\delta_n
\le 
\| d\wt{F}_x (X)\|
\le 1 +\frac{1}{2}\delta_n
\end{equation}
holds where $\delta_n$ is the positive 
constant  defined by Eq.\,\eqref{2017_04_16_sec2.2_new2}. 
\end{lemma}

\begin{proof}
Fix $x \in \B_\pi(o_{q_1}) \setminus \{o_{q_1}\}$. 
Then there is $v \in \Sph^{n-1}_{q_1}$ and $\ell >0$ such that $x= \ell v$. 
We then see, by the proof of \cite[Lemma 3.7]{KT}, 
that \begin{equation}\label{2017_05_09_lem1_2}
d\wt{F}_{x} (av) = d\wt{F}_{\ell v} (av)= a \sigma (v)
\end{equation}
for all $a \in \R$, and 
\begin{equation}\label{2017_05_09_lem1_3}
d\wt{F}_{x} (w) 
= d\wt{F}_{\ell v} (w)
= d\sigma_{v} (w)
\end{equation}
for all 
$w \in \Sph^{n-2}_{v}
:=\{w \in T_{v} (\Sph^{n-1}_{q_1})\,|\, \|w\|=1\}
$. 
Since 
\[
\|d\sigma_{v} (w)\|
= 
\|
I_{q_1}^{(\sigma^{q_1}_{p_1}(v))} (w)
+ 
d\sigma_{v} (w)-I_{q_1}^{(\sigma^{q_1}_{p_1}(v))} (w)
\|,
\]
combining 
the triangle inequality and Lemma \ref{2017_04_16_Key_LEMMA} gives   
\begin{equation}\label{2017_05_09_lem1_4}
1-\frac{1}{2}\delta_n
\le 
\left\|
d \sigma_{v} 
(w)
\right\|
\le 
1+\frac{1}{2}\delta_n
\end{equation}
for all $w \in \Sph^{n-2}_{v}$ 
where 
$I_{q_1}^{(\sigma^{q_1}_{p_1}(v))}:
T_{q_1}M_1\lra T_{q_2}M_2$ is  
the linear isometry defined by 
Eq.\,\eqref{2017_05_10_Isom}, 
and 
$\sigma^{q_1}_{p_1}$ 
is the inverse of the diffeomorphism 
$\sigma^{p_1}_{q_1}:\Sph_{p_1}^{n-1}\lra \Sph_{q_1}^{n-1}$ 
defined by Eq.\,\eqref{2017_05_14_diffeom}. 
Fix $X \in T_{x}(\B_\pi(o_{q_1}))$ with $\|X\|=1$. 
Then there are $w_0 \in \Sph^{n-2}_v$ 
and $\alpha, \beta \in \R$ such that 
$X= \alpha v + \beta w_0$. 
Eqs.\,\eqref{2017_05_09_lem1_2} and \eqref{2017_05_09_lem1_3} imply 
\begin{equation}\label{2017_05_09_lem1_5}
\|
d\wt{F}_{x} (X)
\|^2
= \alpha^2 + \beta^2\|d\sigma_v(w_0)\|^2.
\end{equation}
Since $\alpha^2 + \beta^2 =1$ and $\delta_n \in (0,1)$, 
we see, by Eq.\,\eqref{2017_05_09_lem1_4}, that 
\begin{equation}\label{2018_03_18_lem1_6}
\Big(1-\frac{\delta_n}{2}\Big)^2  
\le 
\alpha^2 + \beta^2
\Big(
1-\frac{\delta_n}{2}
\Big)^2
\le
\alpha^2 + \beta^2\|d\sigma_v(w_0)\|^2
\end{equation}
and that 
\begin{equation}\label{2018_12_14_lem1_new1}
\Big(1+\frac{\delta_n}{2}\Big)^2  
\ge 
\alpha^2 + \beta^2
\Big(
1+\frac{\delta_n}{2}
\Big)^2
\ge
\alpha^2 + \beta^2\|d\sigma_v(w_0)\|^2.
\end{equation}
Substituting Eq.\,\eqref{2017_05_09_lem1_5} 
for Eqs.\,\eqref{2018_03_18_lem1_6} 
and \eqref{2018_12_14_lem1_new1}, 
we get the 
assertion in this lemma.
$\qedd$
\end{proof}

\begin{lemma}\label{2017_05_09_lem1}
For any $y, z \in \B_\pi(o_{q_1})$, 
\begin{equation}\label{2017_05_09_lem1_new1_1}
\Big(
1 -\frac{1}{2}\delta_n 
\Big)
\|y-z\|
\le \|\wt{F}(y)- \wt{F}(z)\| 
\le
\Big(
1 + \frac{1}{2}\delta_n 
\Big)\|y-z\|
\end{equation}
holds. In particular
\begin{equation}\label{2017_05_09_lem1_1}
\bL(\wt{F})\le 1+\delta_n
\end{equation}
holds where $\bL(\wt{F})$ denotes 
the bi-Lipschitz constant of $\wt{F}$ defined by 
\[
\bL(\wt{F})
:=
\inf
\{
L \ |\,L^{-1}\|y-z\| 
\le \|\wt{F} (y)-\wt{F}(z)\|
\le L\,\|y-z\| \ {\rm for \ all} \ y,z \in \B_\pi(o_{q_1})
\}. 
\]
\end{lemma}

\begin{proof}
Fix $x \in \B_\pi(o_{q_1}) \setminus \{o_{q_1}\}$. 
Let $Y \in T_{\wt{F}(x)}(\B_\pi (o_{q_2}))$ with $Y\not=0$. 
Since $d\wt{F}_x$ is bijective, 
there is $X\in T_{x}(\B_\pi (o_{q_1}))$ with $X\not=0$ satisfying $d\wt{F}_x (X)=Y$. 
Since $\|d\wt{F}_x (X)\|=\|Y\|$, 
the left side one in the 
inequality \eqref{2017_05_09_lem1_7} 
gives $(1 -\delta_n/2)\|X\| 
\le 
\|d\wt{F}_x (X)\|=\|Y\|$, 
and hence 
\begin{equation}\label{2017_05_09_lem1_8}
\|X\| 
\le
\Big(
1 -\frac{1}{2}\delta_n 
\Big)^{-1}\|Y\|
\end{equation}
holds. Since $(d\wt{F}_x)^{-1}= (d\wt{F}^{-1})_{\wt{F}(x)}$, 
we have $(d\wt{F}^{-1})_{\wt{F}(x)} (Y) =X$. 
From Eq.\,\eqref{2017_05_09_lem1_8} we thus have 
\begin{equation}\label{2017_05_09_lem1_9}
\|(d\wt{F}^{-1})_{\wt{F}(x)} (Y)\| =\|X\| 
\le \Big(1 -\frac{1}{2}\delta_n\Big)^{-1}\|Y\|.
\end{equation}

We first prove the left side 
one in the assertion \eqref{2017_05_09_lem1_new1_1}. 
Fix $\tilde{y}, \tilde{z} \in \B_\pi(o_{q_2})$. 
We can assume $\tilde{y} \not= \tilde{z}$ in this aim. 
Let 
$\tilde{v}:= (\tilde{z}-\tilde{y})/\|\tilde{z}-\tilde{y}\|$ 
and $a:=\|\tilde{z}-\tilde{y}\|$. 
The geodesic segment 
$\wt{\gamma} :[0,a]\lra \B_\pi(o_{q_2})$ 
emanating from $\tilde{y}$ to $\tilde{z}$ is 
then given by 
$\wt{\gamma} (t):= \tilde{y}+ t \tilde{v}$. 
Since $\wt{F}^{-1}$ is 
Lipschitzian, and since $\dot{\wt{\gamma}}(t)=\tilde{v}$, we see, 
by Eq.\,\eqref{2017_05_09_lem1_9}, that 
\begin{align}\label{2017_05_09_lem1_10}
\|\wt{F}^{-1} (\tilde{z}) -\wt{F}^{-1} (\tilde{y})\| 
&= 
\Big\|
\int_0^a \frac{d (\wt{F}^{-1}\circ \wt{\gamma})}{dt}(t)
\,dt
\Big\|
=
\Big\| 
\int_0^a (d\wt{F}^{-1})_{\wt{\gamma}(t)}
(\tilde{v})\,dt
\Big\|\notag\\[1mm]
&\le 
\int_0^a 
\| 
(d\wt{F}^{-1})_{\wt{\gamma}(t)}
(\tilde{v}) 
\|dt \le 
\Big(1 -\frac{1}{2}\delta_n\Big)^{-1}\int_0^a 
\| 
\tilde{v}
\|dt\notag\\[1mm]
&=
\Big(1 -\frac{1}{2}\delta_n\Big)^{-1}\int_0^a\,dt
=
\Big(1 -\frac{1}{2}\delta_n\Big)^{-1}\|\tilde{z}-\tilde{y}\|.
\end{align}
Set $y:=\wt{F}^{-1}(\tilde{y})$ and 
$z:=\wt{F}^{-1}(\tilde{z})$. 
Eq.\,\eqref{2017_05_09_lem1_10} then 
shows the desired inequality. 
An analogous argument gives 
the right side one 
in the inequality \eqref{2017_05_09_lem1_new1_1}.\par  
We finally prove Eq.\,\eqref{2017_05_09_lem1_1}. 
Since $\delta_n \in (0,1)$, we have
\[
1-\frac{1}{2}\delta_n -\frac{1}{1+\delta_n} = \frac{\delta_n (1-\delta_n )}{2(1+\delta_n )} >0, 
\]
and hence 
$(1+\delta_n)^{-1}< (1- \delta_n/2)$. 
Since $1+ \delta_n/2 < 1+\delta_n$, it follows from Eq.\,\eqref{2017_05_09_lem1_new1_1} 
that 
\begin{equation}\label{2017_05_09_lem1_11}
(1+\delta_n)^{-1}\|y-z\|
\le \|\wt{F}(y)- \wt{F}(z)\| 
\le \left(1+\delta_n \right)\|y-z\|.
\end{equation}
We therefore obtain 
Eq.\,\eqref{2017_05_09_lem1_1} from 
Eq.\,\eqref{2017_05_09_lem1_11}.$\qedd$ 
\end{proof}

By applying the Nash embedding theorem \cite{N} to $M_2$, 
let $M_2$ be isometrically embedded into the Euclidean space $\R^m$ where $m \ge n+1$. 
$F$ then is a Lipschitz map from $M_1$ to 
$M_2 \subset \R^m$.\par 
For any $\ve>0$ sufficiently small 
let $\wt{F}_\ve$ be the standard convolution 
of $\wt{F}$, i.e., $\wt{F}_\ve(x):=\int_{\R^n} \wt{F}(y) \rho_\ve (x-y)dy$ 
where the mollifier $\rho_\ve$ 
near $o_{q_1}$, 
and we identify $T_{q_1}M_1$ with $\R^n$. 
Substituting 
Eq.\,\eqref{2017_04_16_sec2.2_new2} 
for $\delta_n$ in 
Eq.\,\eqref{2017_05_09_lem1_1}, 
we have 
\begin{equation}\label{2017_05_12_final_1}
\bL(\wt{F})^2 
\le 
1+ \Big\{ \frac{8}{\pi}(n-1)\Big\}^{-\frac{1}{2}}. 
\end{equation}
In virtue of 
Eq.\,\eqref{2017_05_12_final_1} 
we can apply the proof of \cite[Theorem 5.1]{K} 
to $\wt{F}_\ve$, and hence 
we see, by the proof, that $\wt{F}_\ve$ 
is an immersion from some open ball $\B_a (o_{q_1})\subset \B_\pi (o_{q_1})$ into $\B_\pi (o_{q_2})$. 
Let $B_a (q_1):=\exp_{q_1}\B_a (o_{q_1})$. 
Define the map $F_\ve^{(q_1)}$ from 
$B_a (q_1)$ into $M_2$ by 
$
F_\ve^{(q_1)}
:= \exp_{q_2} 
\circ
\wt{F}_\ve
\circ 
\{\exp_{q_1}^{-1}|_{B_a (q_1)}\}
$. 
From the definition of $\wt{F}$, 
we see that $F_\ve^{(q_1)}$ is a smooth approximation of $F$ on $B_a (q_1)$. 
It is clear that $F_\ve^{(q_1)}$ is 
an immersion on $B_a (q_1)$.\par  
Let $g:M_1\lra \R$ be 
a smooth function satisfying $0\le g \le1$ on $M_1$, $g\equiv 1$ on $\ol{B_r(q_1)}$, 
and $\supp g \subset B_R(q_1)$ where 
$0<r < R<a$. Define the map $F_\ve:M_1\lra \R^m$ 
by $F_\ve:= (1-g)F + g F_\ve^{(q_1)}$. 

\begin{lemma}\label{2017_05_14_lem}
For any $x_1 \in M_1$, $(dF_\ve)_x$ is injective 
for an $\ve>0$ sufficiently small.
\end{lemma}

\begin{proof}
From the definition of $F_\ve$ we see 
$F_\ve = F_\ve^{(q_1)}$ on $\ol{B_r(q_1)}$ and 
$F_\ve =F$ on $M_1\setminus \supp g$, 
and hence $F_\ve$ is a local diffeomorphism 
on $\ol{B_r(q_1)}\cup (M_1\setminus \supp g)$. 
Since $\wt{F}$ is smooth on $\B_\pi(o_{q_1})\setminus \{o_{q_1}\}$, 
the definition of the differential 
of a smooth map shows that 
$\wt{F}_\ve$ uniformly converges to $\wt{F}$ 
on $\ol{\B_R(o_{q_1})} \setminus \B_r(o_{q_1})$ 
in the $C^1$-topology by letting 
$\ve\downarrow 0$. 
Since 
$dF_\ve^{(q_1)}
= 
d\exp_{q_2}\circ d\wt{F}_\ve 
\circ d \exp_{q_1}^{-1}
$ 
and 
$
dF
= 
d\exp_{q_2}\circ d\wt{F} \circ d \exp_{q_1}^{-1}
$ where 
note that 
$F$ is differentiable on 
$M_1\setminus \{q_1\}$, 
we see, by the argument above, 
that 
$F_\ve^{(q_1)}$ uniformly converges to $F$ 
on $\ol{B_R(q_1)} \setminus B_r(q_1)$ 
in the $C^1$-topology 
by letting $\ve\downarrow 0$. 
Since
\[
F_\ve- F 
= g (F_\ve^{(q_1)} -F) 
= g (\exp_{q_2}\circ \wt{F}_\ve \circ \exp_{q_1}^{-1}
- \exp_{q_2}\circ\wt{F} \circ \exp_{q_1}^{-1})
\]
on $M_1$, and since  
\[
(dF_\ve)_x (v) -dF_x(v)
= dg_x(v)(F_\ve^{(q_1)}(x)-F(x))
+ g (x)\{(dF_\ve^{(q_1)})_x (v) - dF_x(v)\}
\]
for all $v \in T_xM_1$ ($x \in M_1 \setminus\{q_1\}$), 
the second argument above shows 
that $F_\ve$ uniformly converges to $F$ 
on $\ol{B_R(q_1)} \setminus B_r(q_1)$ 
in the $C^1$-topology 
by letting $\ve\downarrow 0$. 
Since $F$ is diffeomorphic on $M_1\setminus \{q_1\}$, the third argument above implies 
that for any $x\in \ol{B_R(q_1)} \setminus B_r(q_1)$, 
$(dF_\ve)_x$ 
is injective for an $\ve>0$ sufficiently small, 
and hence for any $x\in M_1$, 
$(dF_\ve)_x$ 
is too for such an $\ve>0$.
$\qedd$ 
\end{proof}

Since $M_2$ is isometrically embedded into $\R^m$, 
the tubular neighborhood theorem 
(cf. \cite{H}, \cite{L}) via 
the normal exponential map $\exp^\perp:TM_2^\perp\lra \R^m$ shows 
that there is a constant $\mu >0$ such that 
$\exp^\perp$ is a diffeomorphism from 
an open neighborhood 
$\cU_{\mu} (O(TM_2^\perp))$ of the zero 
section $O(TM_2^\perp)$ 
onto an $\cU_{\mu}(M_2)$ 
of $M_2$ in $\R^m$ 
where the two sets are given by 
$
\cU_{\mu} (O(TM_2^\perp))
:=
\{
X\in TM_2^\perp\,|\,\|X\|<\mu\}$ 
and 
$
\cU_{\mu}(M_2):= \exp^\perp 
[\cU_{\mu} (O(TM_2^\perp))]$. 
Since $\exp^\perp|_{\cU_{\mu} (O(TM_2^\perp))}$ 
is bijective, for any $y \in \cU_{\mu}(M_2)$ 
there is a unique point 
$(x, v) \in \cU_{\mu} (O(TM_2^\perp))$ 
such that $y= \exp^\perp (x,v)$. 
For such a pair $(y, (x,v))$ we thus have the smooth projection $\pi_{M_2}:\cU_{\mu}(M_2) \lra M_2$ given by 
$\pi_{M_2}(y) = \pi_{M_2}(\exp^\perp (x,v)):= x$. 
Since $M_1$ is compact, 
we see, by the definition of $F_\ve$ and 
the proof of Lemma \ref{2017_05_14_lem}, that 
$\lim_{\ve\downarrow 0}\|F_\ve (p)- F(p) \|=0$ 
for all $p \in M_1$, which implies 
$F_\ve (M_1) \subset \cU_{\mu}(M_2)$ 
for an $\ve>0$ sufficiently small. 
We can now define the smooth map $\psi_\ve: M_1\lra M_2$ 
by 
\[
\psi_\ve := \pi_{M_2} \circ F_\ve
\]
for an $\ve>0$ sufficiently small.

\begin{lemma}\label{2018_03_21_lem3.4}
$\psi_{\ve_0}$ is 
a smooth immersion for an $\ve_0>0$ sufficiently small.
\end{lemma}

\begin{proof}
Fix a sufficiently small $\ve >0$ so that 
$\psi_\ve$ is defined. 
Since $\dim {\rm Im} (dF_\ve)_p = n$ ($p \in M_1$) by Lemma \ref{2017_05_14_lem}, 
we shall show below that for each $p \in M_1$, 
\begin{equation}\label{2018_03_21_2}
\ra ((d\pi_{M_2})|_{{\rm Im} (dF_\ve)_p}) =n. 
\end{equation}
As we have noted in the proof of Lemma \ref{2017_05_14_lem}, 
$F_\ve = F_\ve^{(q_1)}$ on $\ol{B_r(q_1)}$ and 
$F_\ve =F$ on $M_1\setminus \supp g$, 
and hence we have
$
\psi_\ve (\ol{B_r(q_1)}) 
= F_\ve (\ol{B_r(q_1)}) 
\subset M_2
$ and 
$
\psi_\ve (M_1\setminus \supp g) 
= 
F_\ve (M_1\setminus \supp g) 
\subset M_2$. 
So it is sufficient to show that 
Eq.\,\eqref{2018_03_21_2} 
holds for all $p \in \ol{B_R(q_1)}\setminus B_r(q_1)$:  We first see,  
by the definition of $\pi_{M_2}$, that 
\begin{equation}\label{2018_10_11_new1}
\Ker (d\pi_{M_2})_{F(p)}
\cap 
\Im dF_p =\{o_{F(p)}\}
\end{equation}
for all $p \in \ol{B_R(q_1)}\setminus B_r(q_1)$ where $o_{F (p)}$ denotes the origin of 
$T_{F (p)}\R^m$. As we have seen in the proof of Lemma \ref{2017_05_14_lem}, 
$F_\ve$ uniformly converges to $F$ 
on $\ol{B_R(q_1)} \setminus B_r(q_1)$ 
in the $C^1$-topology by letting 
$\ve\downarrow 0$. 
From Eq.\,\eqref{2018_10_11_new1} 
and 
the similar argument 
as in \cite[Section 5.2]{Kondo2018} we see that 
for an $\ve_0 \in (0, \ve]$ sufficiently small,   
\begin{equation}\label{2018_03_21_3}
\Ker (d\pi_{M_2})_{F_{\ve_0} (p)} 
\cap 
\Im (dF_{\ve_0})_p 
=\{o_{F_{\ve_0} (p)}\}
\end{equation}
for all $p \in \ol{B_R(q_1)}\setminus B_r(q_1)$ 
where $o_{F_{\ve_0} (p)}$ denotes the origin of 
$T_{F_{\ve_0} (p)}\R^m$. Eq.\,\eqref{2018_03_21_3} shows that 
Eq.\,\eqref{2018_03_21_2} 
holds for all $p \in \ol{B_R(q_1)}\setminus B_r(q_1)$ and an $\ve_0>0$ sufficiently small, 
which completes the proof.$\qedd$
\end{proof}

Finally we will show that $\psi_{\ve_0}$ is a global diffeomorphism from $M_1$ onto $M_2$ 
for an $\ve_0 >0$ sufficiently small: Fix ${\ve_0} >0$ sufficiently small so that 
$\psi_{\ve_0}$ is a smooth immersion. Since $\psi_{\ve_0} (M_1)\subset M_2$ is compact, 
and since $M_2$ is Hausdorff, 
$\psi_{\ve_0} (M_1)$ is closed in $M_2$. 
Since $\psi_{\ve_0}$ is a local homeomorphism 
on $M_1$ by Lemma \ref{2018_03_21_lem3.4}, 
$\psi_{\ve_0} (M_1)$ is open in $M_2$. 
$\psi_{\ve_0} (M_1)$ is now open and closed in $M_2$, and hence $\psi_{\ve_0} (M_1) =M_2$, i.e., $\psi_{\ve_0}$ 
is surjective. Since $\psi_{\ve_0}^{-1}(V)$ is closed in $M_1$ for all closed sets $V$ in $M_2$, 
the compactness of $M_i$ ($i=1,2$) shows that 
$\psi_{\ve_0}^{-1}(V)\subset M_1$ and $V\subset M_2$ are compact, which 
implies that $\psi_{\ve_0}$ is a proper map, in particular, is a covering map. Since $M_2$ is simply connected, 
$\psi_{\ve_0}$ is injective. 
Therefore, $\psi_{\ve_0}$ is a global diffeomorphism from 
$M_1$ onto $M_2$.$\qedd$

\end{document}